\documentclass[a4paper,10pt]{amsart}

\usepackage{amsmath,amssymb,amsthm,a4wide,tikz}

\newtheorem*{thm}{Theorem}
\newtheorem{theorem}{Theorem}
\newtheorem*{proposition}{Proposition}
\newtheorem*{lemma}{Lemma}
\newtheorem*{conjecture}{Conjecture}

\title{lower bounds on nodal sets of eigenfunctions\\ via the heat flow}
\author{Stefan Steinerberger}
\address{Mathematisches Institut, Universit\"at Bonn, Endenicher Allee 60, 53115 Bonn, Germany}

\begin{document}
\begin{abstract}
We study the size of nodal sets of Laplacian eigenfunctions on compact Riemannian manifolds
without boundary and recover the currently optimal lower bound by comparing the heat flow of the eigenfunction
with that of an artificially constructed diffusion process. The same method should apply
to a number of other questions; we use it to prove a sharp result saying that a nodal domain 
cannot be entirely contained in a small neighbourhood of a 'reasonably flat' surface and recover
an older result of Cheng. The arising concepts can be expected to have many more connections to 
classical theory and we pose some conjectures in that direction.
\end{abstract}

\maketitle

\section{Introduction}
We consider a compact $n-$dimensional $C^{\infty}-$manifold $(M,g)$ without boundary. Writing $\Delta_g$ for the Laplace-Beltrami operator, we
are interested in structural properties of Laplacian eigenfunctions 
$$ -\Delta_g u = \lambda u.$$
A natural object of study is the measure of the nodal set
$$ Z = \left\{x \in M: u(x) = 0\right\}.$$
An old conjecture of Yau \cite{yau} states that $|Z| \sim \lambda^{\frac12}$. For real-analytic $(M,g)$ this
was proven by Donnelly \& Fefferman \cite{df}, however, in the general $C^{\infty}-$case even $n=2$ is still open.
In two dimensions, the best bounds are 
$$ \lambda^{\frac12} \lesssim |Z| \lesssim \lambda^{\frac34},$$
where the lower bound is by Br\"{u}ning \cite{b} and Yau (unpublished), independently, and the upper bound is due to Donnelly \& Fefferman \cite{d2}
(with another proof given by Dong \cite{do}). Exponentially decaying lower bounds were given by Hardt \& Simon \cite{hs} via a frequency function
approach for a very general class of equations. Asymmetry results on nodal domains due to Donnelly \& Fefferman \cite{df} and Chanillo \& Muckenhoupt \cite{chan}
imply polynomial bounds. The first progress in recent times (i.e. a merely linearly decaying polynomial bound) is due to Sogge \& Zelditch \cite{sz1}
$$ |Z| \gtrsim \lambda^{\frac{7-3n}{8}},$$
although, as was later pointed out, an earlier result by Mangoubi \cite{m} can be combined with the isoperimetric
inequality to yield
$$ |Z| \gtrsim \lambda^{\frac{3-n}{2}-\frac{1}{2n}}.$$
Currently, the best bound is the following.
\begin{thm} The volume of nodal sets satisfies
$$ \mathcal{H}^{n-1}\left(\left\{x \in M: u(x) = 0\right\}\right) \gtrsim \lambda^{\frac{3-n}{4}}.$$
\end{thm}

This was first proven by Colding \& Minicozzi \cite{c}. Subsequently, different proofs were given by Hezari \& Sogge \cite{hso}, 
Hezari \& Wang \cite{hw} (for $n \leq 5$) and Sogge \& Zelditch \cite{sz2}. The arguments tend to be either local estimates on small balls in the style of Donnelly-Fefferman
or global integral formulae. It is the purpose of this paper to give a new local approach exploiting the fact that
a Laplacian eigenfunction behaves nicely under the heat flow. The approach is fully self-contained with the exception of our using
a global inequality due to Sogge \& Zelditch \cite{sz2} 
$$  \frac{\|u\|_{L^1(M)}}{\|\nabla u\|_{L^{\infty}(M)}} \gtrsim \lambda^{-\frac{n+1}{4}},$$
which is known to be sharp on spherical harmonics. We also sketch a variant of our proof that comes
to rely on $\|u\|_{L^1(M)} \gtrsim  \lambda^{\frac{1-n}{4}}\|u\|_{L^{\infty}(M)}$ (also used
by Sogge \& Zelditch \cite{sz1}), which is easily seen to be equivalent because of $\|\nabla u\|_{L^{\infty}(M)} \sim \lambda^{1/2}\|u\|_{L^{\infty}(M)}.$   \\

As a by-product we show that, for $c$ sufficiently small, a nodal domain cannot be contained in a $c\lambda^{-1/2}-$neighbourhood of
a sufficiently flat $(n-1)-$dimensional surface in $M$. The statement is easily seen to be sharp because there are eigenfunctions on the flat torus $\mathbb{T}^2$ such that
any nodal domains is contained in a $C\lambda^{-1/2}-$neighbourhood of a geodesic of length 1. We also give a sub-optimal local version of the previous local result: 
if two line segments contained in the nodal set are contained in a thin rectangle, then the rectangle has bounded eccentricity. Finally,
we formulate two very hard conjectures in the spirit of our approach whose resolution would imply a slightly sharper version of Yau's conjecture.

\section{Idea and Statement of results}

\subsection{The main idea.} 
The main idea is as follows: the heat equation with a Laplacian eigenfunction as initial data and Dirichlet
conditions on the nodal set has the explicit solution 
$$ (\partial_t - \Delta_g)e^{-\lambda t}u(x)=0.$$
This is also the solution of the heat flow without any boundary conditions, however, we will be working locally. In particular, we have precise control on
the rate of decay of the $L^1-$norm in time. A natural candidate for comparison is the heat equation with
the same initial data but Neumann conditions on the nodal set, which conserves $L^1$.
However, the entire difference between Dirichlet and Neumann heat flow is caused by the existence of the nodal set and if it was
too small, it couldn't account for the difference in behavior.\\

Our proof is not actually using the Neumann solution because it requires some regularity on the boundary
and would necessitate using reflected Brownian motion whose construction is nontrivial around the singular set
$$ \left\{x \in M: u(x) = 0 = \nabla u(x)\right\}.$$
Instead, we choose another way and construct a stochastic process which might be interesting in itself: 
for small times it acts as a diffusion but as time grows the process converges back to initial data. We don't
expect any serious obstacles if one were to work with actual Neumann solutions, some
remarks on how to proceed are sketched in the last section of the paper.

\subsection{Bounds on nodal sets.}
We give a new proof of the currently optimal lower bound on the length of nodal sets.
\begin{theorem} We have
$$ \mathcal{H}^{n-1}\left(\left\{x \in M: u(x) = 0\right\}\right) \gtrsim \lambda^{\frac{3-n}{4}}.$$
\end{theorem}
The (purely local) proof will give a sum consisting of local terms over all nodal domains $D$, the 
lower bound is then implied by a global inequality due to Sogge \& Zelditch \cite{sz2} 
$$ \lambda \sum_{D}{\frac{\|u\|_{L^1(D)}}{\|\nabla u\|_{L^{\infty}(D)}}} \geq \lambda \frac{\|u\|_{L^1(M)}}{\|\nabla u\|_{L^{\infty}(M)}} \gtrsim \lambda^{\frac{3-n}{4}}.$$
We expect the argument to be applicable to more general diffusion processes and possibly other 
questions about Laplacian eigenfunctions. A natural question is whether the argument can be
extended to eigenfunctions of the fractional Laplacian if one were to define it as the symbol
associated with the L\'{e}vy jump process.

\subsection{Geometry of nodal sets.} 
One example is the shape of nodal domains, where we briefly describe a simple geometric result.
It deals with the question whether nodal domains can be contained
in a small neighbourhood of a 'flat' surface of codimension 1. In two dimensions, the statement reduces
to a statement mentioned by Mangoubi \cite{m2} and to a theorem of Hayman \cite{hay},
however, in this generality it seems to be new.\\

Let $\Sigma \subset M$ be an arbitrary smooth $(n-1)-$dimensional surface. We ask whether a nodal
domain can be contained in a small neighbourhood of $\Sigma$. The $\varepsilon-$neighbourhood of
a generic geodesic (being itself as 'flat' as possible) on the torus $\mathbb{T}^2$ already coincides with the entire torus -- we therefore need to
place some restrictions on $\Sigma$ for the question to be meaningful. Using $d_{g}(\cdot, \cdot)$
to denote the geodesic distance, we call $\Sigma$ admissible up to distance $r$ if
$$ \forall x \in M: d_{g}(x, \Sigma) \leq r \quad \implies  \quad \#\left\{y \in \Sigma: d(x,y) = d(x, \Sigma)\right\} = 1.$$
This precludes the scenario of dense geodesics and implies that $\Sigma$ is essentially flat
at length scales smaller than $r$. Our next theorem states that a nodal domain cannot
be much flatter than the wavelength $\lambda^{-1/2}$ in any direction.

\begin{theorem} There is a constant $c > 0$ depending only on $(M,g)$ such that if 
$\Sigma \subset M$ is admissible up to distance $\lambda^{-1/2}$, then no nodal domain
can be a subset of the $c\lambda^{-1/2}-$neighbourhood of $\Sigma$.
\end{theorem}
As mentioned above, the function $u(x) = \mbox{Re}\exp(i\sqrt{\lambda} x)$ on $\mathbb{T}^2$
endowed with the flat metric has all its nodal domains contained in a $0.5 \lambda^{-1/2}$
neighbourhood of a geodesic of length 1 (being admissible up to $r=0.5$) and the example 
easily generalizes to higher dimensions.\\

It is not difficult to see (from the proof) that the statement can be generalized to the 
union of admissible sets assuming they are sufficiently transversal at points of intersection
and assuming the complement of the union does not have small connected components. Indeed, 
the proof immediately carries over to the following classical theorem (where the inradius
is defined as the radius of the largest ball fully contained in the domain).
\begin{thm}[Hayman, \cite{hay}] There exists a constant $c \geq 900^{-1}$ such that for any simply
connected domain $\Omega \subset \mathbb{R}^2$ with inradius $\rho$
 $$ \lambda_1(\Omega) \geq \frac{c}{\rho^2}.$$
\end{thm}
Note that the assumption of being simply connected is crucial but can be relaxed provided the
'holes' aren't too small - this is the case for Laplacian eigenfunctions, where the lack
of simple connectivity comes from other nodal domains, which cannot be too small themselves.

\subsection{Avoided crossings.} We now assume the manifold to be two-dimensional
and give a simple result to illustrate the method, which is not optimal but has a very simple proof.
The nodal set consists of lines -- if two line segments from contained in
the nodal set are very close to each other, they can either intersect and create a singular point 
or be close to each other along a short line segment. 
\begin{center}
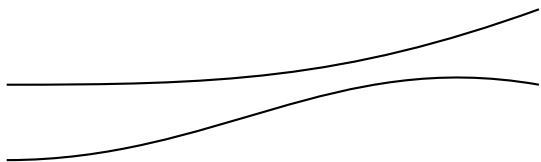
\begin{figure}[h!]
\begin{tikzpicture}
\draw[thick] (0,0) to [out=0,in=200] (7,1);
\draw[thick] (0,-1) to [out=0,in=170] (7,0);
\end{tikzpicture}
\caption{Two nodal lines almost crossing.}
\end{figure}
\end{center}

In the second case, sometimes termed 
'avoided crossings' \cite{mo}, it is natural to ask for bounds on the length of the line segment. 
Polynomial bounds follow from early work by Donnelly \& Fefferman \cite{df3}, which has been
recently refined by Mangoubi \cite{m2}. If $(M,g)$ is a $C^{\infty}-$surface he shows (among other
things) that two nodal lines cannot be at distance $\lambda^{-1}$ along a line segment of length 
$\lambda^{-1/2}\sqrt{\log{\lambda}}$. We give a relatively short argument showing the same result 
up to a loss of a factor $\sqrt{\log{\lambda}}$.

\subsection{Cone condition.} In $n=2$, an older result of Cheng \cite{cheng} (or, implicitely, a
1955 theorem of Lipman Bers) implies that a nodal domain satisfies an interior cone condition with 
opening angle $\lambda^{-1/2}$. This bound is
attained for spherical harmonics, in particular the function which locally looks like 
$$ \mbox{Re}(x+iy)^{\sqrt{\lambda}}$$
around the origin. We give a very simple proof of the statement using results on
hitting probabilities for Brownian motion in cones. The argument is very general and also applies
in higher dimensions and different shapes.

\subsection{Open questions.} The main idea of getting information on the nodal set by constructing
diffusion processes which deviate in behavior at the boundary -- because its infinitesimally
generated particles are absorbed/reflected differently -- relies on having the right concepts.
As it turns out, a rather crucial concept is the notion of heat content, which seems to
have been studied by many people but in a predominantly asymptotic sense (see, e.g. \cite{vdb}).
We conjecture an isoperimetric principle for the heat content and show how such a principle 
would imply yet another proof of the bound on the length of nodal sets. It also allows for a natural refinement of
Yau's conjecture and seems to be the suitable for proving a Lieb-type generalization of Hayman's
theorem -- this is discussed in the last section of the paper.\\

 As for notation, the symbols $\lesssim$ and 
$\sim$ will always denote absolute constants depending only on the manifold $(M,g)$.

\section{Proof of the Theorem 1.}
\subsection{Heat content.} Given an open subset $N \subset M$, we use $p_t(x)$ to denote
the solution to the following heat equation
\begin{align*}
 (\partial_t - \Delta_g)p_t(x) &= 0 \quad \quad \quad x  \in N\\
p_t(x) &= 1 \quad \quad \quad x  \in \partial N \\
p_0(x) &= 0 \quad \quad \quad x  \in  N.
\end{align*}
The Feynman-Kac formula implies that this can be understood as the probability that a Brownian motion particle started in $x$ will hit the
boundary within $t$ units of time. The quantity
$$\int_{N}{p_{t}(x)dx}$$
is called the heat content of $N$ at time $t$. It can be seen as a 'soft' measure of boundary size -- for large times the function will be
roughly of size 1 in the entire domain and all information on the size of the boundary will be lost. However, within $t$ units of time a typical Brownian motion particle travels a distance of $\sim t^{1/2}$.
This can be immediately seen with Varadhan's large deviation formula \cite{var}
$$ \lim_{t \rightarrow 0}{-4t\log K(t,x,y)} = d(x,y)^2,$$
where $K(\cdot,\cdot,\cdot)$ is the heat kernel on the manifold $(M,g)$. We also refer to standard Gaussian estimates on the heat kernel 
(see for example the book of Grigoryan \cite{gri}). For small times $t$, the function $p_t(x)$ is essentially supported in a neighbourhood of size $\sim t^{1/2}$
from the boundary and is superexponentially decaying after that: this is what yields a connection to the size of the boundary. 
 In particular, we will prove that for nodal domains
 $N \subset M$ as $t \rightarrow 0$
$$\int_{N}{p_{t}(x)dx} \sim \sqrt{t}\mathcal{H}^{n-1}(\partial D).$$
Much more precise results are known for domains with $C^{\infty}-$boundary. Around 1970, Greiner \cite{greiner} and Seeley \cite{seeley}
independently showed that there exists an asymptotic series 
$$\int_{N}{p_{t}(x)dx} \sim \sum_{n=1}^{\infty}{a_{n}(N)t^{\frac{n}{2}}} \qquad \mbox{as}~t \rightarrow 0^+.$$
There has been some interested in expressing the initial coefficients in terms of geometric quantities of $N$: this
can indeed be done in the smooth context (see, for example, the survey of Gilkey \cite{gilkey}).

\subsection{Definitions.} Let $D$ be an arbitrary nodal domain. Without loss of generality, 
we assume the eigenfunction $u(x)$ to be positive within $D$: otherwise consider $-u(x)$.
Given $u(x)$, we define $v(t,x)$ as solution to the heat equation with $u(x)\big|_{D}$ as initial 
data and Dirichlet condition on the boundary. We set 
$$v(t,x) := e^{-\lambda t}u(x)$$
and note that $v(t,x)$ then solves
\begin{align*}
 (\partial_t - \Delta_g)v(t,x) &= 0 \quad \mbox{on}~D \setminus \left\{u(x)=0\right\}\\
v(t,x) &= 0 \quad \mbox{on}~\left\{u(x)=0\right\}\\
v(0,x) &= u(x) \quad \mbox{on}~D.
\end{align*}

A relevant classical concept is the Feynman-Kac formula for the Dirichlet problem (see e.g. Taylor \cite{tay}),
which allows to rewrite a deterministic diffusion process as an expectation over the behavior of random
variables: given an open domain $\Omega \in \mathbb{R}^n$,
$f \in L^2(\Omega)$, $x \in \Omega$ and $t > 0$, then
$$ (e^{t\Delta_D}f)(x) = \mathbb{E}_{x}(f(\omega(t))\psi_{\Omega}(\omega, t)),$$
where $t > 0$ is arbitrary, $\omega(t)$ denotes an element of the probability space of Brownian 
motions starting in $x$, $\mathbb{E}_x$
is to be understood with regards to the measure of that probability space and
$$ \psi_{\Omega}(\omega,t) = \begin{cases}
                              1 \qquad &\mbox{if}~\omega([0,t]) \subset \Omega \\
                              0 \qquad &\mbox{otherwise.}
                             \end{cases}$$
Here we see the connection with the heat content: for any point $x \in \Omega$ and any $t>0$
$$ \mathbb{E}_x(\psi_{\Omega}(\omega,t)) = 1-p_t(x).$$
For reasons that will become apparent in the proof, for $f \in C^{\infty}_0(\Omega)$ we define a second 'diffusion'
operator $\Xi$ via 
$$(e^{t\Xi}f)(x) := \mathbb{E}_{x}(f(\omega(t))\psi_{\Omega}(\omega, t)) + \mathbb{E}_{x}(1-\psi_{\Omega}(\omega, t))f(x).$$
This operator is initially smoothing but ceases being so as time progresses. Indeed, if $\Omega \subset M$ such that $M \setminus \Omega$
is open, then if $\Xi$ is adapted to $\Omega$ it is easy to see that for every $x \in \Omega$
$$ \lim_{t \rightarrow \infty}{(e^{t\Xi}f)(x)} = f(x).$$
Finally, we claim that
$$ \int_{\Omega}{e^{t\Xi}f dx} = \int_{\Omega}{f dx},$$
which is equivalent to
$$ \int_{\Omega}{p_t(x)u(x)dx} = \int_{\Omega}{(1-e^{t\Delta_D})u(x)dx}.$$
A stochastic argument would be to say that among paths not leaving the domain, it is
equally likely to start in a point $x$ and end in a point $y$ than the other way around -- 
a statement that follows from the symmetry $K(t,x,y) = K(t,y,x)$ of the heat kernel.

\subsection{A Comparison Lemma.} We are interested in comparing the behavior of the Dirichlet
solution $e^{t\Delta_D}u$ with the behavior of $e^{t\Xi}u$ on a fixed nodal domain $D$, where we assume
without loss of generality that $u\big|_{D} \geq 0$ (otherwise: consider $-u(x)$). It is obvious from
the definition that
$$ e^{t\Xi}u \geq e^{t\Delta_D}u.$$
\begin{lemma}
 There exists a constant $C > 0$ depending only on $(M,g)$ such that
$$ \int_{D}{e^{t\Xi}u(x) - e^{t\Delta_D}u(x)dx} \leq C  \|\nabla u\|_{L^{\infty}} t^{1/2}
 \int_{D}{p_t(x) dx}.$$
\end{lemma}
\begin{proof}
Our starting point is given by the definition of $ e^{t\Xi}$, which gives the pointwise equation
$$ e^{t\Xi}u - e^{t\Delta_D}u = p_t(x) u(x).$$
This pointwise equation is the key to proving the inequality: $u(x)$ will grow with increasing
distance to the boundary but we can use the trivial estimate coming from the mean-value theorem
$$ u(x) \leq d(x,\partial D)\|\nabla u\|_{L^{\infty}}.$$
Thus
$$ \int_{D}{p_t(x) u(x)dx} \leq \|\nabla u\|_{L^{\infty}} \int_{D}{d(x,\partial D) p_t(x)dx}.$$
The function $p_t(x)$ behaves like a smooth cutoff-function around the boundary since
it follows from Feynman-Kac that $p_t(x)$ is the probability of a Brownian motion started
in $x$ hitting the boundary within $t$ units of time. Thus, since a Brownian motion particle
travels on average a distance $\sim t^{1/2}$ (and larger distances have a superexponentially
decaying tail)
$$ p_t(x) \leq c_1 e^{-c_2 d(x,\partial D)^2/t}$$
for two absolute constants $c_1, c_2 > 0$ depending only on $(M,g)$. This last step is completely
equivalent to Varadhan's short-time asymptotic (which will be used more directly below in another
step of the proof) and could have been replaced by that. 
This integral now contains a product of a $1-$Lipschitz function and a superexponentially decaying
function, which starts to decay rapidly around $d(x, \partial D) \sim t^{1/2}$. Therefore, we get that for some $C > 0$ depending only on $(M,g)$
that
$$ \|\nabla u\|_{L^{\infty}} \int_{D}{d(x,\partial D) p_t(x)dx} \leq 
 \|\nabla u\|_{L^{\infty}} \int_{D}{C t^{1/2}p_t(x)dx}.$$
\end{proof}

\textit{Remark.} The manifold is compact: therefore the estimate of $p_t(x)$ being localized within a $t^{1/2}-$neighbourhood of the nodal set
is too rough for large time and only really accurate for $t \lesssim \lambda^{-1/2}$, which is precisely the time-scale on which the
Lemma will be ultimately applied.

\subsection{Conclusion.} Fix again an arbitrary nodal domain $D$ and we assume
again without loss of generality that $u(x)\big|_{D} \geq 0.$
The heat equation, the comparison lemma and the $L^1-$conservation of $\Xi$ give
\begin{align*}
 e^{-\lambda t}\int_{D}{u(x) dx} &= \int_{D}{e^{t\Delta_{D}}u(x) dx} \\
&\geq \int_{D}{e^{t\Xi}u(x) dx} -  C t^{1/2}  \|\nabla u\|_{L^{\infty}}\int_{D}{p_t(x)dx} \\
&= \int_{D}{u(x) dx} -  C t^{1/2}  \|\nabla u\|_{L^{\infty}}\int_{D}{p_t(x)dx}
\end{align*}
Therefore
$$ \int_{D}{p_t(x)dx} \gtrsim \frac{1-e^{-\lambda t}}{t^{1/2}}\frac{\|u\|_{L^1(D)}}{\|\nabla u\|_{L^{\infty}(D)}}.$$
The result now follows once we show that up to constants depending on the manifold
$$ \lim_{t \rightarrow 0^+}{\frac{1}{\sqrt{t}}\int_{D}{p_t(x)dx}} \sim \mathcal{H}^{n-1}(\partial D).$$
This result will follow from showing that for $x \in D$ with the property that $\partial D$
is smooth in a $t^{1/2}-$neighbourhood of $x$, we have
$$c_3 e^{-c_4 d(x,\partial D)^2/t} \leq p_t(x) \leq c_1 e^{-c_2 d(x,\partial D)^2/t},$$
i.e. that the upper bound coming from the heat kernel is of the right order  -- this is an easy consequence
of Varadhan's large deviation formula  \cite{var} . The second part is to show that almost all parts of the boundary (up to a set of small Minkowski dimension) are smooth, which follows from a recent
result of Cheeger, Naber \& Valtorta \cite{ch} (their result is actually stronger than we require).\\

The involved quantities are essentially local and so is our argument: it is known that the critical set
$$ \left\{x \in D: u(x) =  |\nabla u(x)| = 0 \right\}$$
has $(n-1)-$dimensional \textit{Minkowski} measure 0 - recent results by Cheeger, Naber \& Valtorta
even give bounds on its $(n-2)-$dimensional Minkowski measure. 
Fix a small $t > 0$ and cover the $\partial D$ with cubes of side length $t^{-1/2}$. We call a cube \textit{regular}
if it does not contain an element of the singular set and \textit{singular} otherwise. From the result above, it
follows that the number of singular cubes is of order $o(t^{(-(n-1)/2)})$.\\

\textit{Regular cubes.} Let $t > 0$ be fixed and let $Q$ be a regular cube. Once a cube is regular, there is no need
to further refine it as time tends to 0: the nodal set is given as a $C^{\infty}-$hypersurface.
\begin{center}
\begin{figure}[h!]
\begin{tikzpicture}
\draw (0,0) --(3,0);
\draw (3,0) --(3,3);
\draw (0,3) --(3,3);
\draw (0,3) --(0,0);
\node at (2,2) {$D$};
\node at (2,1) {$D^c$};
\draw[thick] (-3,0) to [out=0,in=200] (5,3);
\end{tikzpicture}
\caption{A regular cube with its nodal set.}
\end{figure}
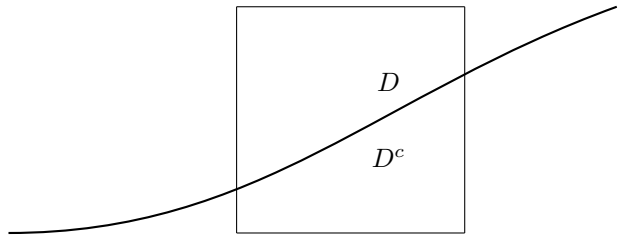
\end{center}

 To see this, we fix a second time-parameter $z > 0$
and study the behavior of
$$ \lim_{z \rightarrow 0^+}{\frac{1}{\sqrt{z}}\int_{D \cap Q}{p_z(x)dx}}$$
as $z \rightarrow 0^+$. As $z \rightarrow 0^+$, the function $p_z(x)$ becomes
concentrated in smaller and smaller neighbourhoods of the surface. Since the
surface is locally $C^{\infty}$, we may treat it as a flat hyperplane of
codimension 1 embedded in $\mathbb{R}^n$ and use the explicit heat kernel 
in $\mathbb{R}^n$ to compute the relevant quantity. The fact that this 
argument is actually stable under small perturbations of the surface
follows immediately from Varadhan's large deviation formula $$ \lim_{z \rightarrow 0^+}{-4z\log K(z,x,y)} = d(x,y)^2.$$
In turn this implies that as $z \rightarrow 0$, we have that $p_z(x)$ is of size $p_z(x) \sim 1$ for $x$ in a $z^{1/2}-$neighbourhood
of $\partial D$ and vanishes superexponentially at larger distances. From the (local) $C^{\infty}-$regularity of the set $\partial D \cap Q$, we get that
$$\lim_{z \rightarrow 0^+}{\frac{1}{\sqrt{z}}\int_{D \cap Q}{p_z(x)dx}} \geq c \mathcal{H}^{n-1}(\partial D \cap Q)$$
for some constant $c > 0$ depending only $(M,g)$. \\

\textit{Singular cubes.}  It remains to show that the error introduced by those cubes containing an element of the
singular set is small: it is not enough to note that their relative proportion is small because they are weighted with
a factor $t^{-1/2}$, which becomes singular for small times. Using $0 \leq p_t(x) \leq 1$ gives
$$ \left| \frac{1}{\sqrt{t}}\int_{D_{\mbox{sing}}}{p_t(x)dx}\right| \leq \frac{1}{\sqrt{t}}|D_{\mbox{sing}}|.$$
$D_{\mbox{sing}}$ consists of $o(t^{(-(n-1)/2)})$ cubes of side-length $t^{1/2}$, therefore
$$ \frac{1}{\sqrt{t}}|D_{\mbox{sing}}| \leq 
\frac{1}{\sqrt{t}}t^{n/2}o(t^{(-(n-1)/2)}) = o(1).$$
Improved estimates on the Minkowski dimension actually imply a faster rate of decay
but these are not necessary for the conclusion of the argument. $\qed$

\section{Thin nodal sets}

In this section, we give a proof of Theorem 2. It is based on the fact that at scale $\sim \lambda^{-1/2}$ 
the neighbourhood of a surface admissible up to $\lambda^{-1/2}$ behaves like the neighbourhood of a 
hyperplane in $\mathbb{R}^n$, which allows for problems to be reduced to well-known one-dimensional facts
(indeed, our notion of 'admissible' is chosen such that this is true). It is perhaps easiest to understand
the proof first for $n=2$, where all key elements are already present: for $n=2$ an admissible surface is
merely a curve with curvature $\kappa \leq \lambda^{-1/2}$. The main idea is that a Brownian motion particle is equally likely to wander
in every direction: in a a 'squeezed nodal domain', it will hit the boundary too often.

\begin{proof}[Proof of Theorem 2.] We consider the heat flow with Dirichlet conditions on the nodal domain $D$
\begin{align*}
 (\partial_t - \Delta_g)v(t,x) &= 0 \quad \mbox{on}~D \setminus \left\{u(x)=0\right\}\\
v(t,x) &= 0 \quad \mbox{on}~\left\{u(x)=0\right\}\\
v(0,x) &= u(x).
\end{align*}
The other case being identical, we assume $u(x)$ to be positive in $D$. We start by proving a
statement showing the existence of some point $x \in D$ such that Brownian motions starting in
$x$ are not very likely to hit the boundary: from physical intuition it is not surprising that
these points should be close to those points, where the eigenfunction assumes its maximum 
and this guides our argument; for some fascinating results in that direction, we refer
to Grieser \& Jerison \cite{jerison}. We will prove that
$$ \forall~t>0 \qquad \inf_{x \in D}{p_t(x)} \leq 1-e^{-\lambda t}$$
by showing the following slightly stronger statement
$$ \forall x \in D\qquad  u(x) = \|u\|_{L^{\infty}(D)} \implies p_t(x) \leq 1-e^{-\lambda t}.$$
Given a $x \in D$ with $u(x) = \|u\|_{L^{\infty}(D)}$, we see using the heat equation and Feynman-Kac
\begin{align*}
 e^{-\lambda t} \|u\|_{L^{\infty}(D)} &= e^{-\lambda t}u(x) 
= \mathbb{E}_x(u(\omega(t))\psi_{D}(\omega(t))) \\
&\leq \|u\|_{L^{\infty}(D)}\mathbb{E}_x(\psi_{D}(\omega(t))) = \|u\|_{L^{\infty}(D)}(1-p_t(x)).
\end{align*}
This proves the claim.\\

We now set the time to be $t= \lambda^{-1}$. It remains to show that choosing $c$ small enough, we can derive
a contradiction to this bound on $p_t(x)$. Take a small $c > 0$ and a $c\lambda^{-1/2}-$neighbourhood of the
admissible surface $\Sigma$. Assume $D$ to be a nodal domain fully contained in that set and let $x \in D$
be such that $u(x) = \|u\|_{L^{\infty}(D)}$. The statement we need to contradict is $p_{\lambda^{-1}}(x) \leq 1-e^{-1}$
and we will do so but suitably bounding the probability of leaving the $c\lambda^{-1/2}-$neighbourhood of
$\Sigma$ from below to achieve a contradiction: if the nodal domain was contained in a small neighbourhood
of $\Sigma$, it will hit the boundary pretty definitely and we expect $p_{\lambda^{-1}}(x)$ to be arbitrarily
close to 1 as $c$ becomes small. The remainder of the proof consists in making this precise.
\begin{center}
\begin{figure}[h!]
\begin{tikzpicture}
\node at (3,1.5) {$D^c$};
\node at (6,-1) {$D^c$};
\node at (-0.3,0) {$\Sigma$};
\node at (4,0.7) {$D$};
\node at (5,0) {$D$};
\draw[ultra thick] (0,0) to [out=0,in=200] (7,1);
\draw[thick, dashed] (0,-1) to [out=0,in=200] (7,0);
\draw[thick, dashed] (0,1) to [out=0,in=200] (7,2);
\draw[thick] (0,0.8) to [out=0,in=200] (3,0.9) to [out=15,in=190] (7,1.5);
\draw[thick] (0,-0.8) to [out=10,in=180] (3,-0.5) to [out=0,in=190] (7,0.1);
\end{tikzpicture}
\caption{An example in two dimensions: the surface $\Sigma$ (thick), its $c\lambda^{-1/2}-$neighbourhood (dashed) and the boundary of the nodal domain $D$.}
\end{figure}
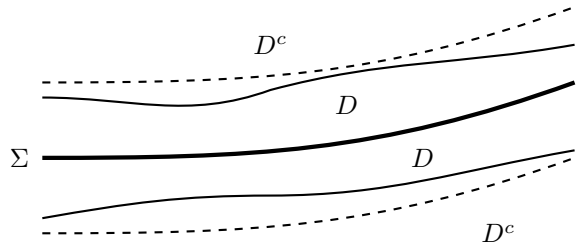
\end{center}
Let now $x \in D$. A Brownian motion starting in $x$ has, at any point inside $D$, at most $n-1$ 'good' directions in which it can wander unhindered
and at least 1 'bad' direction: it is only allowed to wander in direction of the normal of $\Sigma$ for a very short distance before impacting on the boundary. For $0 < c \ll 1$, the 
curvature of the surface plays hardly any role: we can assume the surface to be a flat hyperplane.\\

If $\Sigma = \mathbb{R}^{n-1}$, then $\mbox{dist}(\omega(t), \Sigma)$ behaves like a one-dimensional Brownian motion $B(t)$
and we have
 $$ p_{\lambda^{-1}(x)} \geq \mathbb{P}\left( \sup_{0 < s < \lambda^{-1}}{B(t)} > c\lambda^{-1/2} \right).$$
This quantity, however, is well-understood and the reflection principle (see e.g. \cite{ka}) implies
$$ \mathbb{P}\left( \sup_{0 < s < \lambda^{-1}}{B(t)} > c\lambda^{-1/2} \right)= 2 \mathbb{P}\left( B(\lambda^{-1}) > c\lambda^{-1/2} \right).$$
However, $B(\lambda^{-1})$ is just a random variable following a normal distribution with mean $\mu = 0$ and variance $\sigma = \lambda^{-1}$. By symmetry
$$ 2 \mathbb{P}\left( B(\lambda^{-1}) > c\lambda^{-1/2} \right) =  \mathbb{P}\left( \left|B(\lambda^{-1})\right|> c\lambda^{-1/2} \right)$$
and by bounding the normal distribution by its maximal value
$$ \mathbb{P}\left( \left|B(\lambda^{-1})\right|> c\lambda^{-1/2} \right) \geq 1 
- \int_{-c\lambda^{-1/2}}^{c\lambda^{1/2}}{\frac{1}{\sqrt{2\pi}}\frac{1}{\lambda^{1/2}}dx} = 1 - \sqrt{\frac{2}{\pi}}c.$$
This yields a contradiction for $c < \sqrt{\pi}/(\sqrt{2}e)$ in the case of $\Sigma = \mathbb{R}^{n-1}$. A perturbative version of this
argument applies to more general curved surfaces (with a potentially smaller $c$).
\end{proof}

\textit{Question.} It could be interesting to study isoperimetric principles for these types of problems. Let $\Sigma \subset \mathbb{R}^n$ be a $C^{\infty}-$surface,
$x \in \Sigma$ and $\varepsilon > 0$. Is the probability of a Brownian motion leaving a $\varepsilon-$neighbourhood of $\Sigma$ minimized in the flat case
$\Sigma = \mathbb{R}^{n-1}$ and the Brownian motion starting in a point $x \in \Sigma$? Do minimal surfaces play a distinguished role?

\section{Avoided crossings}
Let $(M,g)$ be as above and, additionally, two-dimensional. Fix some $\alpha > 1/2$. We define an \textit{avoided crossing} as follows: let $T$ be a geodesic connecting two points 
$a,b \in M$ and $D$ be a nodal domain containing $T$. We say that $D$ avoids a crossing if there exists a $\lambda^{-\alpha}$-neighbourhood of $T$ such that if a point in the nodal domain $x \in D$ 
is at distance $\lambda^{-\alpha}$ from the geodesic $d(x, T) = \lambda^{-\alpha}$, then it must be close to one of the of the endpoints $a$ or $b$
$$\min(d(a,x), d(b,x)) \leq \lambda^{-\alpha}.$$
\begin{center}
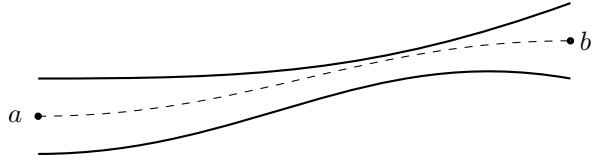
\begin{figure}[h!]
\begin{tikzpicture}
\draw [ultra thick] (0,-0.5) circle [radius=0.02];
\node at (-0.3,-0.5) {$a$};
\draw [ultra thick] (7,0.5) circle [radius=0.02];
\node at (7.2,0.5) {$b$};
\draw[dashed] (0,-0.5) to [out=0,in=180] (7,0.5);
\draw[thick] (0,0) to [out=0,in=200] (7,1);
\draw[thick] (0,-1) to [out=0,in=170] (7,0);
\end{tikzpicture}
\caption{Two nodal lines almost crossing: a nodal domain contained in a small neighbourhood of a geodesic (dashed).}
\end{figure}
\end{center}

Alternatively, between $a$ and $b$ every point $x$ of the nodal domain is at distance at most $\lambda^{-\alpha}$ from the geodesic line segment $T$. If $\alpha$ is big, then for
this to be possible, $a$ and $b$ need to be very close together.
\begin{proposition} If $D$ avoids a crossing, then 
$$ d(a,b) \leq C\lambda^{1/2-\alpha}\log{\lambda}$$
for some constant $C<\infty$ depending only on $(M,g)$. 
\end{proposition}
We have a nontrivial statement precisely if $\alpha > 1/2$ -- as we have seen above, a nodal line may well be contained in the $\lambda^{-1/2}$ neighbourhood of a geodesic.
The result could be optimal up to the logarithmic factor.

\begin{proof} We consider the set
$$ D \cap  \left\{y \in M: d(y,T) \leq \lambda^{-\alpha}\right\}$$
and cover it with $N$ squares of scale $\lambda^{-\alpha} \times \lambda^{-\alpha}$, which we call $R_1, R_2, \dots, R_N$ and where the enumeration is such that
$R_{i}$ borders on $R_{i-1}$ and $R_{i+1}$. Our goal is to prove the upper bound $N \lesssim \sqrt{\lambda}\log{\lambda}$ on the number of squares. This will then imply the result since
$d(a,b) \sim \lambda^{-\alpha}N$.
\begin{center}
\begin{figure}[h!]
\begin{tikzpicture}
\draw (0,1) --(0,-1);
\draw (0,-1) --(2,-1);
\draw (2,-1) --(2,1);
\draw (0,1) --(2,1);

\draw (2,1) --(2,-1);
\draw (2,-1) --(4,-1);
\draw (4,-1) --(4,1);
\draw (2,1) --(4,1);

\draw (4,1) --(4,-1);
\draw (4,-1) --(6,-1);
\draw (6,-1) --(6,1);
\draw (4,1) --(6,1);

\draw (6,1) --(6,-1);
\draw (6,-1) --(8,-1);
\draw (8,-1) --(8,1);
\draw (6,1) --(8,1);

\draw (8,1) --(8,-1);
\draw (8,-1) --(10,-1);
\draw (10,-1) --(10,1);
\draw (8,1) --(10,1);

\draw (10,1) --(10,-1);
\draw (10,-1) --(12,-1);
\draw (12,-1) --(12,1);
\draw (10,1) --(12,1);

\draw[thick] (0,0.8) to [out=0,in=200] (3,0.4) to [out=15,in=190] (12,0.8);
\draw[thick] (0,-0.8) to [out=10,in=180] (3,-0.1) to [out=0,in=190] (12,0.1);
\end{tikzpicture}
\caption{An almost crossing and a covering with squares.}
\end{figure}
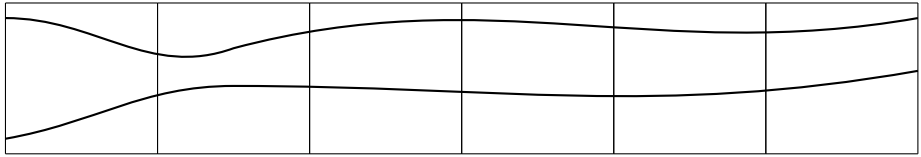
\end{center}
Let us quickly illustrate the main idea: we consider the evolution of the heat equation 
with Dirichlet boundary with the eigenfunction as
initial data for very short time $t = \lambda^{-2\alpha}$. The explicit solution implies
that this time is too short for any real change to happen, the function is almost static
on that time scale: we write
\begin{align*}
 (\partial_t - \Delta_g)v(t,x) &= 0 \quad \mbox{on}~D \setminus \left\{u(x)=0\right\}\\
v(t,x) &= 0 \quad \mbox{on}~\left\{u(x)=0\right\}\\
v(0,x) &= u(x).
\end{align*} 
Since $\alpha > 1/2$ and $t = \lambda^{-2\alpha}$, we have
$$ v\left(\lambda^{-2\alpha} ,x\right) = e^{-\lambda^{1-2\alpha}}u(x) \sim u(x).$$
Let us now consider a square in the covering and a Brownian motion particle: it moves
a distance of $\sim \lambda^{-\alpha}$: it will thus likely either enter another square
or impact on the boundary of the nodal domain (and both events will happen with a
probability uniformly bounded away from 0). However, the effect of particles impacting
on the boundary implies a loss of the $L^1-$norm, which we know is not there -- therefore
this loss is being counterbalanced by the surviving particles carrying back larger mass. \\

A Brownian motion particle started in $R_i$ for time $t = e^{-\lambda^{\alpha}}$ can either
impact on the boundary with probability $p_{b} > 0$ (bounded away uniformly from 0), can end
up in any of the other squares with probability $p_{ij}$ or exit the entire covered domain
entirely with probability $p_{ie}$. Note that
$$ p_{b} + \sum_{j=1}^{N}{p_{ij}} + p_{ie} = 1.$$
Using the Feynman-Kac formula, this implies
$$ e^{-\lambda^{1-2\alpha}}\sup_{x \in R_{i}}{|u(x)|} \leq  p_{ie}\|u\|_{L^{\infty}(M)} + \sum_{j=1}^{N}{p_{ij}\sup_{x \in R_{j}}{|u(x)|}} \qquad \qquad \qquad (\diamond)$$
Note that the decay of the heat kernel implies
\begin{align*}
 p_{ij} &\lesssim \exp{\left(-|i-j|^2\right)} \\
 p_{ie} &\lesssim \exp{\left(\min\left(i^2, (N-i)^2\right)\right)}.
\end{align*}
Let us now prove the statement by contradiction: we assume from now on that $N \gtrsim \sqrt{\lambda}\log{\lambda}$.\\

Pick some $N/3 \leq i \leq 2N/3$. Then the contribution gained from exiting the entire domain is negligible since
 $$p_{ie}\|u\|_{L^{\infty}(M)} \lesssim \lambda^{\frac{n-1}{4}}e^{-N^2/100} \leq \lambda^{\frac{n-1}{4}}\left(\frac{c}{\lambda}\right)^{\lambda \log{\lambda}}.$$
Let us now imply the inequality $(\diamond)$ for $i = \left\lfloor N/2 \right\rfloor$. It now implies that there is
a rectangle $j$ such that
$$ \exp{\left(c_1|\left\lfloor N/2 \right\rfloor-j|^2\right)}\sup_{x \in R_{\left\lfloor N/2 \right\rfloor}}{|u(x)|} \leq \sup_{x \in R_{j}}{|u(x)|}$$
for some universal constant $c_1$ depending only on $(M,g)$. We want to iterate this inequality
several times to show that $u$ has to be much bigger than $\sup_{x \in R_{\left\lfloor N/2 \right\rfloor}}{|u(x)|}$
at some other place. If $j \leq N/3$ or $j \geq 2N/3$,
we quit, otherwise we reiterate the procedure until the index leaves the range $\left\{N/3,
N/3 + 1, \dots, 2N/3\right\}$. The worst case is that for each $i$ the inequality holds true
with $j = i+1$ in which case we still have
$$ (1+c_2)^{N/6}\sup_{x \in R_{N/2}}{|u(x)|} \leq \|u\|_{L^{\infty}(M)} \lesssim \lambda^{\frac{n-1}{4}}.$$
At the same time we have the vanishing order estimate due to Donnelly \& Fefferman and thus
$$ \sup_{x \in R_{\left\lfloor \frac{N}{2} \right\rfloor}}{|u(x)|}  \gtrsim \inf_{i}\sup_{x \in R_{i}}{|u(x)|} \gtrsim \left(\frac{1}{\lambda^{c_2 \alpha}}\right)^{\sqrt{\lambda}}$$
for some $c_2 > 0$ depending only on the manifold, which combined implies
$$ N \lesssim \sqrt{\lambda}\log{\lambda}.$$
\end{proof}

\textit{Remark.} This example gives a heat-flow approach to the phenomenon that elliptic equations in narrow domains
exhibit rapid growth -- a classical elliptic version of this principle also appears in the work of Mangoubi \cite{m2}.

\section{Opening angles}
The purpose of this section is to give a new, short and transparent proof of the following result.
\begin{thm}[Bers, Cheng]
 Let $n = 2$. If $-\Delta u = \lambda u$, then any nodal set satisfies an interior cone condition with opening angle $\alpha \gtrsim \lambda^{-1/2}$.
\end{thm}
The underlying idea is very simple: suppose the opening angle of the cone was very small. Rescaling allows us to study points very close to the apex
and in particular the survival probability of Brownian motions starting from there.
\newcommand{\Emmett}[5]{% points, advance, rand factor, options, end label
\draw[#4] (0,0)
\foreach \x in {1,...,#1}
{   -- ++(#2,rand*#3)
}
node[right] {#5};
}

\begin{center}
 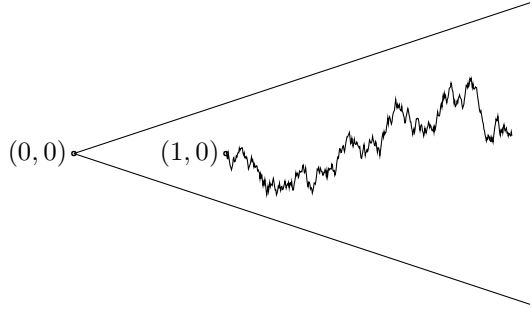
\begin{figure}[h!]
  \begin{tikzpicture}[scale = 0.4]
\draw [ultra thick] (0, 0) circle [radius=0.02];
\node at (-1.2,0) {$(1,0)$};
\draw [ultra thick] (-5, 0) circle [radius=0.02];
\node at (-6.2,0) {$(0,0)$};
\draw (-5,0) --(10,5);
\draw (-5,0) --(10,-5);
\Emmett{470}{0.02}{0.2}{black}{}
\end{tikzpicture}
\caption{A Brownian motion started inside a cone.}
 \end{figure}
\end{center}
Since the opening angle is small, the survival probability is small as well,
which in turn implies fast decay of the eigenfunction close to the apex -- ultimately contradicting the Donnelly-Fefferman estimate on the vanishing order.
\begin{proof} We use the following result, which we found in the book of M\"orters \& Peres \cite[Theorem 7.24]{peres}. Let $W(\alpha)$ denote the cone with opening angle $\alpha$ with vertex in
the origin and symmetric around the $x-$axis and let $B(t)$ be a Brownian motion started in $(1,0)$. Define a stopping time 
$$ T(r) = \inf \left\{t \geq 0: |B(t)| = r\right\}.$$
Then, for $r > 1$,
$$ \mathbb{P}\left(B[0,T(r)] \subset W(\alpha)\right) = \frac{2}{\pi}\arctan\left(\frac{2 r^{\frac{\pi}{\alpha}}}{r^{\frac{2\pi}{\alpha}}-1}\right).$$
We will be using the result in the regime $r \gg 1$, in which case
$$ \mathbb{P}\left(B[0,T(r)] \subset W(\alpha)\right) \sim r^{-\frac{\pi}{\alpha}}.$$
Suppose we are given a point $x_0 \in M$ at the boundary of a nodal domain. The order of vanishing is at most $\sqrt{\lambda}$, meaning
$$ |u(x) - u(x_0)| \geq c_1|x-x_0|^{c_2\sqrt{\lambda}}$$
for some constants $c_1, c_2 > 0$ and $x$ contained in the cone. The cone itself, however, is scaling invariant. Let us again consider
the heat equation
\begin{align*}
 (\partial_t - \Delta_g)v(t,x) &= 0 \quad \mbox{on}~D \setminus \left\{u(x)=0\right\}\\
v(t,x) &= 0 \quad \mbox{on}~\left\{u(x)=0\right\}\\
v(0,x) &= u(x)
\end{align*} 
with the explicit solution $e^{-\lambda t}u(x)$ and let us assume w.l.o.g. that $u \geq 0$ on that nodal domain. Pick now a $x$ inside the cone close to $x_0$. 
We consider the solution of the heat equation for time $t = |x-x_0|$. A typical Brownian motion travels a distance of $t^{1/2} = |x-x_0|^{1/2},$
which is why we set (after rescaling) $r = |x-x_0|^{-1/2}$ in the above theorem. The survival probability of not impacting on the boundary during that time is then given by 
$ \sim |x-x_0|^{\frac{1}{2}\frac{\pi}{\alpha}}.$
The Feynman-Kac formula immediately implies that then
$$ c_1|x-x_0|^{c_2\sqrt{\lambda}}e^{-\lambda|x-x_0|} \leq v(1/\lambda,x) \leq  2|x-x_0|^{\frac{1}{2}\frac{\pi}{\alpha}}\|u\|_{L^{\infty}(M)} =  2|x-x_0|^{\frac{1}{2}\frac{\pi}{\alpha}} \lambda^{\frac{n-1}{4}}.$$
Letting $x \rightarrow x_0$ proves the statement.
\end{proof}

\textit{Remarks.} This proof seems to extend to variety of shapes. Escape probabilities from various regions have been widely studied for Brownian motion (e.g. Ba\~{n}uelos \& Smits \cite{ban} for cones
in higher dimensions) -- such results immediately extend to restrictions on the shape of nodal domains via the above argument.

\section{Comments and conjectures}

 We believe the heat content to be a possibly valuable tool in the further study of Laplacian eigenfunctions; this section studies some further
implications, in particular we conjecture an isoperimetric statement, which would imply another proof of our estimate on the size of nodal domains.

\subsection{Heat content isoperimetry.} The heat content is a very stable notion and well-defined even for very rough domains. We consider
the following statement to be highly plausible.

\begin{conjecture} Let $(M,g)$ be a compact $C^{\infty}-$manifold without boundary. There exists a constant $c > 0$ depending only on $(M,g)$ such that for any open subset $N \subset M$ and all times $t > 0$
$$ \int_{N}{p_t(x)dx} \leq c\mathcal{H}^{n-1}(\partial N)\sqrt{t},$$
where the Hausdorff measure is understood to be $\infty$ if undefined.
\end{conjecture}
\textbf{Remarks.}
\begin{enumerate}
 \item Extremizers of the inequality need to have a smooth boundary: small irregularities in the boundary increase the surface measure but have very limited impact on the left-hand side. 
It would be interesting to understand the relation between the nature of extremizers and geometric properties of the manifold. Is there a connection to Cheeger sets?\\
 \item If the domain $N$ has the property that there is a real number $r > 0$ such that each point $x \in N$ is contained in a ball of radius $r$ (possibly
centered around another point), then the two quantities should be comparable up to $t \sim r^2$. If $N$ is a nodal domain of the Laplacian, the Faber-Krahn inequality implies that the inradius
is at most $\sim \lambda^{-1/2}$ and therefore $t \sim \lambda^{-1}$ is the maximum time up to which we expect the quantities to be
comparable.
\end{enumerate}
In particular, we conjecture that for a nodal domain both quantities are indeed comparable up to $t= \lambda^{-1}$. If this could be shown, it would immediately imply
$$ \mathcal{H}^{n-1}\left(x \in M: u(x) = 0\right) \lesssim \lambda^{\frac{1}{2}}.$$

\subsection{Isoperimetry implies the main statement.}  Assuming the conjecture to be true, the second remark suggests that the maximum viable time for its application to a nodal
domain without loss is given by $t = \lambda^{-1}$. We ignore possible issues arising in its construction (see, e.g. Bass \& Hsu \cite{bh}) and assume the existence of the reflected Brownian motion on the nodal domain. 

\begin{theorem} Assuming heat content isoperimetry and existence of reflected Brownian motion, we have
$$ \mathcal{H}^{n-1}\left(\left\{x \in M: u(x) = 0\right\}\right) \gtrsim  \lambda^{\frac{1}{2}} \sum_{D}{\frac{\|u\|_{L^1(D)}}{\|u\|_{L^{\infty}(D)}}} \gtrsim \lambda^{\frac{3-n}{4}},$$
where the sum ranges over all nodal domains $D$. 
\end{theorem}
\begin{proof} The proof has strong similarities to our previous argument. Again, without loss of generality, we assume $u(x) > 0$ on $D$ and write $e^{t\Delta_D}$ and $e^{t\Delta_N}$ for evolution
under Dirichlet and Neumann data, respectively. Our new comparison estimate is even simpler and states that on a nodal domain $D$
$$ e^{t\Delta_N}u - e^{t\Delta_D}u \leq p_{t}(x)\| u \|_{L^{\infty}(D)}.$$ 
The proof for this comparison statement is easy to sketch: the difference between Dirichlet and Neumann solutions arises from those Brownian motions
hitting the boundary. The difference is maximized if all those particles hitting the boundary arrive in a maximum of $u$ after having been
reflected.  \\

Integrating the comparison at time $t = \lambda^{-1}$ yields
\begin{align*}
 e^{-1}\int_{D}{u(x) dx} &= \int_{D}{e^{ \Delta_D}u(x) dx} \\
&\geq \int_{D}{e^{\Delta_N}u(x) - p_t(x)\|u\|_{L^{\infty}(D)} dx}  \\
&= \int_{D}{u(x) dx} - \|u\|_{L^{\infty}(D)}\int_{D}{p_t(x)dx}.
\end{align*}
Therefore, assuming heat content isoperimetry and using the Sogge-Zelditch inequality
$$ \mathcal{H}^{n-1}(\partial D) \gtrsim \lambda^{1/2}\frac{\|u\|_{L^1(D)}}{\|u\|_{L^{\infty}(D)}}.$$
Summing over all nodal domains and using the Sogge-Zelditch inequality yields the result.
 \end{proof}

\subsection{Yet another proof.} Another variant of the proof is as follows. We assume again the existence of reflected
Brownian motion. The estimate
$$ e^{t\Delta_N}u - e^{t\Delta_D}u \leq C t^{1/2} p_t(x) \|\nabla u\|_{L^{\infty}}$$
follows from studying the difference between Brownian motion reflected and absorbed at the boundary and using the fact that within $t$
units of time a Brownian motion may travel a distance of up to $\sim t^{1/2}$. Then,
\begin{align*}
 e^{-\lambda t}\int_{D}{u(x) dx} &= \int_{D}{e^{\lambda t \Delta_D}u(x) dx} \\
&\geq \int_{D}{e^{\lambda t\Delta_N}u(x) -C t^{1/2} p_t(x)\|\nabla u\|_{L^{\infty}(D)} dx}  \\
&= \int_{D}{u(x) dx} - C t^{1/2}\|\nabla u\|_{L^{\infty}(D)}\int_{D}{p_t(x)dx}.
\end{align*}
In the limit $t \rightarrow 0^+$,
$$ \lim_{t \rightarrow 0^+}{\frac{1}{\sqrt{t}}\int_{D}{p_t(x)dx}} \gtrsim \lambda \frac{\|u\|_{L^1}}{\|\nabla u\|_{L^{\infty}}}$$
and summing over all domains gives with the Sogge-Zelditch inequality the desired result.

\subsection{Geometric structure of nodal sets.} The quantity $p_t(x)$ can be seen as a local measure of the closeness and size of the boundary
for each nodal domain. It seems extremely natural to conjecture the following for the global $p_t(x)$ function (which is defined by demanding that
its restriction to a nodal domain coincides with the local $p_t(x)$ function there).

\begin{conjecture} Let $(M,g)$ be a compact $C^{\infty}-$manifold without boundary. There exists a constant $c > 0$ depending only on $(M,g)$ such that if $p_t(x)$ 
is globally defined with respect to the nodal set of a Laplacian eigenfunction with eigenvalue $\lambda$, then
$$ p_{\lambda^{-1}}(x) > c \qquad \mbox{for all}~x \in M.$$
\end{conjecture}

This conjecture, while seeming likely, should be extremely difficult. In particular, combining it with heat content isoperimetry at time $t = \lambda^{-1}$
immediately would immediately imply one half of Yau's conjecture
$$ \mathcal{H}^{n-1}\left(\left\{x \in M: u(x) = 0\right\}\right) \gtrsim \lambda^{\frac{1}{2}}.$$

\subsection{Heat content, Laplacian eigenvalues and the inradius.}
Given a domain $\Omega \in \mathbb{R}^n$, we can define the first eigenvalue of the domain as
$$ \lambda_{1}(\Omega) = \inf_{f \in H^1_0(\Omega)}{\frac{\int_{\Omega}{\left| \nabla f(x)\right|^2dx}}{\int_{\Omega}{f(x)^2dx}}}.$$
An inequality of the form
$$ \lambda_1(\Omega) \gtrsim r^{-2},$$
where $r$ is the inradius of the domain, is trivial if $n=1$, true for simply connected domains in $n=2$ (Hayman's theorem) and false for $n \geq 3$. Indeed, in dimensions $n \geq 3$ it
is possible to introduce very thin spikes making the inradius small but having little overall influence on the eigenvalue. However, Lieb \cite{l} in a celebrated paper has shown that 
Hayman's theorem 'essentially' generalizes to higher dimensions: for any domain $\Omega \in \mathbb{R}^n$, there is a ball $B$ of radius $r \sim \lambda^{-1/2}$
such that $|\Omega \cap B| \sim |B|$ (the theorem also gives a precise relationship between the implicit constants). We conjecture that this phenomenon persists for the
heat content.
\begin{conjecture} Let $\Omega \in \mathbb{R}^n$ be a bounded open set. If $c_1 > 0$ and all $0 < t < \lambda_1(\Omega)^{-1}$
 $$ \int_{\Omega}{p_t(x)dx} \geq c_1\mathcal{H}^{n-1}(\partial \Omega)\sqrt{t},$$
then there is a ball $B$ of radius $\sqrt{t}$ such that $|B \cap \Omega| \geq c_2|B|$, where $c_2$ depends only on the dimension and $c_1$.
\end{conjecture}
Of course, via Lieb's theorem, this would establish a mutual equivalence between the first Laplacian eigenvalue, the size of balls having a large
intersection with the domain and the time up to which heat content isoperimetry is sharp ($t = \lambda_1(\Omega)^{-1}$).

\end{document}